\documentclass[12pt,a4paper]{article}

\usepackage{theorem,enumerate}
\usepackage{amsmath,latexsym,amssymb,amsfonts, epsfig}
\usepackage{eucal}
\usepackage{mathrsfs}
\usepackage{array}
\usepackage{epstopdf} 

\newcounter{Scounter}
\theorembodyfont{\normalfont\slshape}
\newtheorem{thm}{Theorem}

\newtheorem{cor}[thm]{Corollary}
\newtheorem{prop}[thm]{Proposition}
\newtheorem{definition}[thm]{Definition}
\newtheorem{lemma}[thm]{Lemma}
\newtheorem{ques}[thm]{Question}

\newtheorem{fact}[thm]{Fact}

\newcommand{\PROOF}[1]{\medbreak\noindent\textbf{Proof #1.}\ }

\newcommand{\qed}{{$\quad\square$\vs{3.6}}}

\newcommand{\vs}[1]{\vspace*{#1 mm}}

\bfseries\normalfont

\makeatletter
\def\thanks#1{%
   \footnotemark
   \edef\@tempa{\noexpand\noexpand\noexpand\footnotetext[\the\c@footnote]}%
   \toks@\expandafter{\@thanks}%
   \toks\tw@{{#1}}
   \xdef\@thanks{\the\toks@\@tempa\the\toks\tw@}}
\makeatother

\begin{document}

\title{On Homeomorphically Irreducible Spanning Trees in Cubic Graphs}

\author{
Arthur Hoffmann-Ostenhof\thanks{Technical University of Vienna,
Austria.
Email: {\tt arthurzorroo@gmx.at}}\thanks{This work was supported by the Austrian Science Fund (FWF): P 26686.},
Kenta Noguchi\thanks{Department of Mathematics, 
Tokyo Denki University, 
5 Senjuasahicho, Adachi-ku, Tokyo 120-8551, Japan.
Email: {\tt noguchi@mail.dendai.ac.jp}},
and Kenta Ozeki\thanks{Environment and Information Sciences, Yokohama National University, 79-2 Tokiwadai,
Hodogaya-ku, Yokohama 240-8501, Japan.
Email: {\tt ozeki-kenta-xr@ynu.ac.jp}}\thanks{
This work was in part supported 
by JSPS KAKENHI Grant Number 25871053.}
}
\date{}
\maketitle

\begin{abstract}
A spanning tree without a vertex of degree two is called a Hist which is an abbreviation for homeomorphically irreducible spanning tree. We provide a necessary condition 
for the existence of a Hist in a cubic graph.
As one consequence, we answer affirmatively an open question on Hists
by Albertson, Berman, Hutchinson and Thomassen.
\end{abstract}


\section{Introduction}

All graphs considered here are finite and simple. In a connected graph $G$, a spanning tree which does not have a vertex of degree two is called a \textit{homeomorphically irreducible spanning tree},
or abbreviated a \textit{Hist}.
Several conditions which ensure the existence of a Hist in a graph are known,
see for instance \cite{ABHT, CS, Hill}.
In this paper, we only consider Hists in cubic graphs. For an integer $k$, a connected cubic graph $G$ which contains two disjoint cycles is said to be \textit{cyclically $k$-edge-connected} if deleting any set of at most $k-1$ edges from $G$ does not separate $G$ into two components both of which have a cycle.
The following question was asked in \cite[p.~253]{ABHT}.

\begin{ques}
\label{ABHTques}
Does there exist a cyclically $k$-edge-connected cubic graph without a Hist for each positive integer $k$?
\end{ques}

Note that every Hist $T$ in a cubic graph has only vertices of degree one and three. Hence $E(G)$ has a partition into 
$E(T)$ and the edge set of a union of disjoint cycles. 

Let us call a $2$-regular subgraph $H$ of a connected graph $G$ \textit{non-separating}
if $G-E(H)$ is connected.
For a set $S$ of edges in $G$,
we denote by $\langle S \rangle$ 
the subgraph of $G$ induced by the edges in $S$.
So, the vertex set of $\langle S \rangle$ is the set of end vertices of edges in $S$.
We answer Question 1 
by applying Corollary \ref{HIST}, a corollary of Theorem \ref{general} which turns out to be useful for proving that certain cubic graphs do not have a Hist.

\begin{thm}
\label{general}
Let $G$ be a cubic graph with a Hist $T$, and
let $H= \langle E(G) - E(T)\rangle$. Then $H$ is a non-separating $2$-regular subgraph of $G$
satisfying $|V(H)|= |V(G)|/2 +1 $.
\end{thm}

\begin{proof}
Let $G$ be a cubic graph with a Hist $T$
and let $H= \langle E(G) - E(T)\rangle$. 
Since $V(H)$ is the set of all leaves of $T$
and $G - E(H) = T$,
$H$ is a non-separating $2$-regular subgraph of $G$.

Let $t_1$ be the number of leaves in $T$,
and let $t_3$ be the number of vertices of degree $3$ in $T$.
Since $T$ is a Hist, we have $t_1 + t_3 = |V(G)|$.
On the other hand,
it is easy to see that $t_1 = t_3 + 2$.
(This can be obtained straightforwardly
by the Handshaking Lemma or by induction.
For example, see \cite[Exercise 2.1.23 on p.~70]{West}.)
Therefore,
$|V(G)| = 2 t_1 -2$. 
Since $V(H)$ is the set of all leaves of $T$,
we have $|V(H)|= t_1$.
By using the above equations the proof is completed.\qed 
\end{proof}



\begin{cor}
\label{HIST}
Let $G$ be a bipartite cubic graph. If $G$ has a Hist, 
then $|V(G)| \equiv \ 2  \,\,(\text{mod} \  4)$.
\end{cor}

\begin{proof}
Let $G$ be a bipartite cubic graph with a Hist.
By Theorem \ref{general},
$H= \langle E(G) - E(T)\rangle$ is a non-separating $2$-regular subgraph of $G$
satisfying $|V(H)|= |V(G)|/2 +1 $.
Since $G$ is bipartite,
$|V(H)|$ is even
and hence $|V(G)| \equiv \ 2  \,\,(\text{mod} \  4)$.
\qed

\end{proof}

\noindent
\textbf{Remark:}
Corollary \ref{HIST} implies that
no bipartite cubic graph $G$ 
with $|V(G)| \equiv \ 0  \,\,(\text{mod} \  4)$
has a Hist.
However,
if $G$ is a bipartite cubic graph
with $|V(G)| \equiv \ 2  \,\,(\text{mod} \  4)$, 
then $G$ may or may not have a Hist.
Both cases could happen, see Section \ref{pcHsec}.  
\\

Now we obtain a positive answer to Question \ref{ABHTques} by applying Corollary \ref{HIST}
together with the following proposition.

\begin{prop} \label{constlemma}
For every positive integer $k$,
there exists a cyclically $k$-edge-connected bipartite cubic graph $G$
such that $|V(G)| \equiv \ 0  \,\,(\text{mod} \  4)$.
\end{prop}

Proposition \ref{constlemma} can be directly proved
by considering transitive graphs:
it is known that
for any positive integer $k$,
there are infinitely many vertex-transitive 
bipartite cubic graphs $G$
of girth at least $k$
with $|V(G)| \equiv \ 0  \,\,(\text{mod} \  4)$, see for example \cite{NS3}.
Since the cyclic edge-connectivity of vertex-transitive graph
is equal to its girth (see \cite{NS2}), Proposition \ref{constlemma} holds.
However, since this proof requires several algebraic tools,
we prefer to present an elementary proof which also offers
a new method to construct cubic bipartite graphs
with high cyclic edge-connectivity, see Theorem \ref{inflation} and Lemma \ref{bip} in Section \ref{proofsec}.

In Section \ref{pcHsec},
we show other application of Theorem \ref{general} to plane and toroidal cubic graphs.



\section{Proof of Proposition \ref{constlemma}}
\label{proofsec}

In order to prove Proposition \ref{constlemma},
we use the following fact which can be proved in several ways, for instance, by the probabilistic method (see \cite[Theorems 2.5 and 2.10]{Wormald}) and by the constructive method (see \cite{Egawa}).

\begin{fact}
\label{constfact}
For every positive integer $d$,
there exists a $d$-connected $4d$-regular graph
of girth at least $d$.
\end{fact}

Then we apply the well known concept of an \textit{inflation} (see for instance \cite{FJ}):

\begin{definition}\label{infl}
Let $H$ be a graph and let $G$ be a cubic graph. Then $G$ is called an \textit{inflation} of $H$ if $G$ contains a 
2-factor $F$ consisting of chordless cycles such that the graph obtained from $G$ by contracting each cycle of $F$ to a vertex is isomorphic to $H$.
\end{definition}

If the minimum degree of $H$ is at least $3$, 
then obviously an inflation of $H$ exists,
since one obtains (informally speaking) an inflation of $H$ by expanding every vertex of $H$ to a cycle.
The next theorem guarantees the high cyclic edge-connectivity for each inflation of graphs with high connectivity and girth.

\begin{thm}\label{inflation}
Let $k \geq 3$ and let $H$ be a $k$-connected 
graph with girth at least $k$. Then every inflation of $H$ is cyclically $k$-edge-connected.
\end{thm}

\begin{proof}
Let $G$ be an inflation of $H$.
For each vertex $x \in V(H)$, 
denote the unique cycle of $F$ (as in Definition \ref{infl}) in $G$
corresponding to $x$ by $C_x$.
We say that
a cycle $C$ in $G$ is \textit{transverse}
if there are two distinct vertices $x_1$ and $x_2$ in $H$
with $V(C_x) \cap V(C) \not= \emptyset$ for each $x \in \{x_1,x_2\}$.
Otherwise $C$ is said to be \textit{non-transverse},
that is, $C = C_x$ for some vertex $x \in V(H)$.

Suppose by contradiction that $G$ is not cyclically $k$-edge-connected.
Then $G$ has a set $S$ of edges with
$|S| \leq k-1$ such that $G - S$ 
has precisely two components $D_1$ and $D_2$ both having a cycle.
By taking such a set $S$ as small as possible,
we may assume that $S$ is a matching.

For $i \in \{1,2\}$,
let $D_i^H$ be the subgraph of $H$
induced by the vertex set $$\{ x \in V(H) : V(C_x) \cap V(D_i) \not= \emptyset\}\,.$$
So, $D_1^H$ is obtained from $D_1$ in the following way: 
for each $x \in V(H)$ such that $C_x \cap D_1$ is not the null graph,
where $C_x \cap D_1$ is the maximum common subgraph of $C_x$ and $D_1$,
contract $C_x \cap D_1$ into one vertex and delete all resultant loops.
%
\begin{eqnarray*}
\text{Let} \qquad
S_{\text{V}}^H &=& \{x \in V(H) : E(C_x) \cap S \not=\emptyset\},\\
\text{and} \qquad
S_{\text{E}}^H &=& S \cap E(H) \\
&=& \{e \in S : e \not\in  E(C_x) \ \text{for any $x \in V(H)$}\}.
\end{eqnarray*}
%
Note that $|S^H_V| + |S^H_E| \ \leq \ |S| \ \leq \ k-1$.

Suppose that $V(D_i^H) - S^H_{\text{V}} \not= \emptyset$ for each $i \in \{1,2\}$.
Then $H - S^H_{\text{V}} - S^H_{\text{E}}$ have two components $D_1^H$ and $D_2^H$.
In this case, the number of vertex disjoint paths from a vertex of $D_1^H$ to a vertex of $D_2^H$
is at most $|S^H_{\text{V}}| + |S^H_{\text{E}}| \leq k-1$,
which contradicts by Menger's Theorem that 
$H$ is $k$-connected. 


Therefore,
we may assume without loss of generality that 
$V(D_1^H) - S^H_{\text{V}} = \emptyset$.


Note that $D_1$ contains by assumption a cycle, say $C_1$,
and 
$C_1$ must be transverse (otherwise, $C_1 = C_x$ for some $x \in V(D^H_1)-S^H_{\text{V}}$,
but this contradicts that $V(D^H_1)-S^H_{\text{V}} = \emptyset$).
Thus, $C_1$ corresponds to a closed trail in $D^H_1$, say $C^H_1$.
Since the girth of $H$ is at least $k$ and every closed trail contains a cycle,
we have $|V(C^H_1)| \geq k$,
which is a contradiction to 
the fact that
$V(C^H_1) \subseteq V(D^H_1) \subseteq S^H_{\text{V}}$
and $|S^H_{\text{V}}| \leq k-1$. \qed

\end{proof}

Note that the statement of the above theorem does not hold
if $H$ is only demanded to be $k$-edge-connected.

\begin{lemma}\label{bip}
Let $k \geq 2$ and let $H$ be a $2k$-regular graph. 
Then there exists a bipartite cubic inflation of $H$ with $2k|V(H)|$ vertices.
\end{lemma}

\begin{proof} Since every inflation of $H$ has $2k|V(H)|$ vertices,
it suffices to show that $H$ has a bipartite inflation.

Since each component of $H$ is Eulerian, 
it has an Eulerian orientation, that is, the indegree equals the outdegree for every vertex of $H$.
Then we can expand every vertex $x$ of $H$ to
a cycle $C_x$ to obtain an inflation $G$ with the property that
the oriented edges incident with the vertices of $C_x$ are alternately directed
towards and outwards $C_x$.
See Figure \ref{inflationfig}.
Furthermore,
since each cycle $C_x$ is of length exactly $2k$,
it is possible to extend this partial orientation to an orientation of $G$
(by orienting the edges of each cycle $C_x$)
such that every vertex of $G$ has then either outdegree $3$ or indegree $3$.
This shows a 2-coloring of $G$, and hence $G$ is bipartite. \qed
\end{proof}

\begin{figure}
\centering
\input{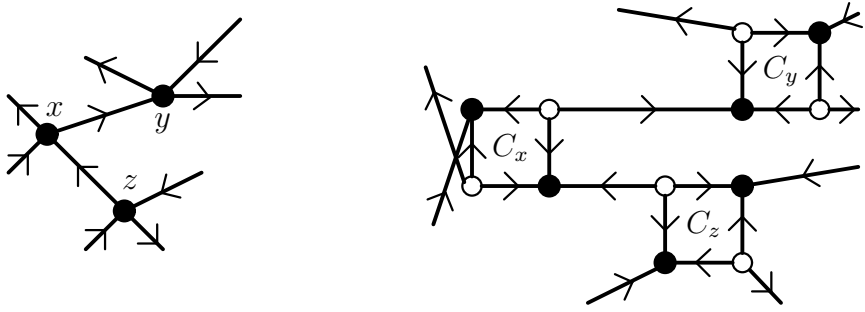}
\caption{A graph $H$ with Eulerian orientation (the left side)
and the bipartite graph $G$ obtained
by an inflation of $H$ (the right side) for the case $k=2$.
In $G$, the vertices with outdegree 3
are represented by white circles,
while the vertices with indegree 3 are represented by black circles.
}
\label{inflationfig}
\end{figure}




\PROOF{of Proposition \ref{constlemma}}
Let $k$ be a positive integer.
By Fact \ref{constfact} there exists a $k$-connected $4k$-regular graph $H$
of girth at least $k$. Since $H$ is $4k$-regular,
it follows from Lemma \ref{bip} that
there exists a bipartite cubic inflation $G$ with $4k|V(H)|$ vertices.
Since $H$ is $k$-connected and has the girth at least $k$, it follows from Theorem \ref{inflation} that $G$ is cyclically $k$-edge-connected,
which completes the proof.
\qed

\section{Hists in plane cubic graphs}
\label{pcHsec}

Let us call a plane cubic graph with a Hist in short a \textit{pcH-graph}. A pcH-graph is by its definition a generalization of a cubic Halin graph (defined in \cite{Hal}) which is a pcH-graph with a Hist such that all the leaves of the Hist induce precisely one cycle. 
It is easy to see that any cubic Halin graph contains a triangle.
In contrast to cubic Halin graphs, pcH-graphs can have girth $4$ or even $5$, see Figure \ref{fullerene_fig}. 
Note that it is NP-complete to determine
whether a plane cubic graph has a Hist, see \cite{Douglas}.
(To be exact, Douglas \cite{Douglas} proved that only for plane graphs of maximum degree at most $3$.
However, replacing each vertex of degree at most $2$ with a certain gadget,
we can easily modify the proof to show the NP-completeness of the Hist problem for plane cubic graphs.)
Since any non-facial cycle of a cubic plane graph is separating,
by restricting Theorem \ref{general} to the planar case we obtain:

\begin{cor}
\label{pH-graph}
Let $G$ be a plane cubic graph with a Hist. 
Then $G$ contains a non-separating $2$-regular subgraph $H$ consisting of facial cycles
such that $|V(H)|=|V(G)|/2 +1 $.
\end{cor}

\begin{figure}[htpb]
\centering\epsfig{file=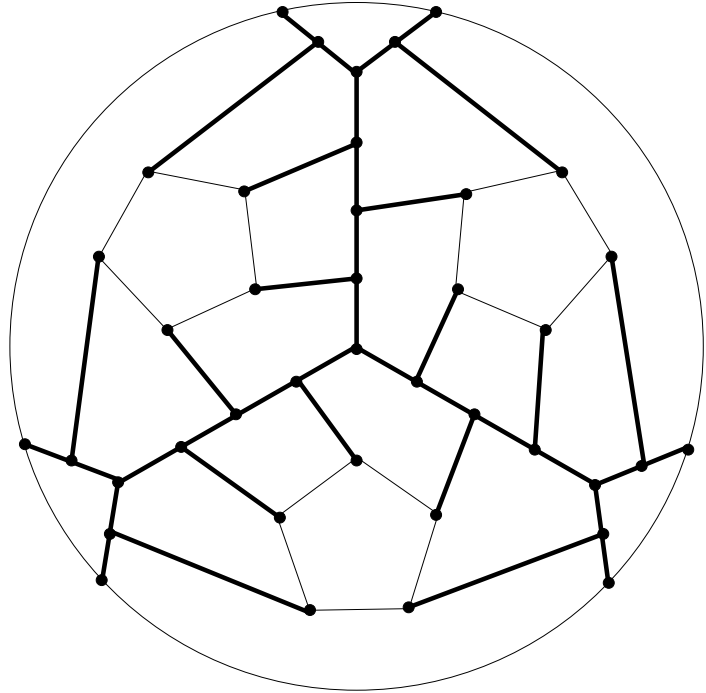, width=3.6 in}
\caption{A fullerene graph with a Hist.} 
\label{fullerene_fig}
\end{figure}

Applying Corollary \ref{pH-graph}, we see for instance that the dodecahedron does not have a Hist, since it has $20$ vertices and every facial cycle has length $5$. 
The dodecahedron belongs to the class of \textit{fullerene graphs},
which are plane $3$-connected cubic graphs with facial cycles of length $5$ and $6$ only, see Figure \ref{fullerene_fig} for an example. Using Corollary \ref{pH-graph} one can prove straightforwardly that other plane cubic graphs, for instance the Buckminster fullerene graph \cite[Figure 9.5.~on p.~211]{GR} and the Grinberg graph \cite[Fig.18.9.~on p.~480]{Bo}, 
do not have a Hist. This should illustrate the usefulness of the above corollary. 
We asked in the first version of this paper whether there are finitely or infinitely many fullerene graphs with a Hist which is answered below.

\begin{thm}
\label{constructionthm}
There are infinitely many fullerene graphs with a Hist.
\end{thm}

\begin{proof}
Let $A$, $B$ and $H$ be the plane graphs shown in Figure \ref{figure1} (every label of a vertex in the figure is shown left above the vertex). By identifying the cycle $C_1$ of $A$ and the cycle $C_2$ of $B$ such that $u$ and $v$ are identified, we obtain a fullerene graph with a Hist which is illustrated in bold edges. In order to construct infinitely many fullerene graphs with Hists, we use the graphs $H_i$ which are defined as follows (during the construction of $H_i$ we keep the bold edges of every copy of $H$ which will then define the edges of the Hist within $H_i$ in the fullerene graph). Firstly, let $H_0$ be a plane cycle of length $18$ and let $H_1\simeq H$.
Then define the graph $H_i$ $(i\ge 2)$ recursively, 
by identifying the outer cycle $C'_1$ in a copy of $H$ and the cycle ($18$-gon) $C'_2$ in $H_{i-1}$ so that $u'$ in $C'_1$ and $v'$ in $C'_2$ are identified.
Now we construct for every nonnegative integer $k$ the fullerene graph $F_k$ with $36k+46$ vertices. We
identify the cycle $C_1$ in $A$ and the outer cycle $C'_1$ in $H_k$
such that $u$ in $C_1$ and $u'$ in $C'_1$ are identified.
Finally, we identify the cycle $C'_2$ in $H_k$ and the outer cycle $C_2$ in $B$
such that $v'$ in $C'_2$ and $v$ in $C_2$ are identified.  
Note that the $12$ shaded faces in Figure \ref{figure1} are pentagons of $F_k$.
It is not difficult to verify that the bold edges in $A$, $B$ and the bold edges of $H_i$ induce a Hist in $F_k$. \qed
\end{proof}

\begin{figure}[h]
 \centering
 \includegraphics[width=16cm]{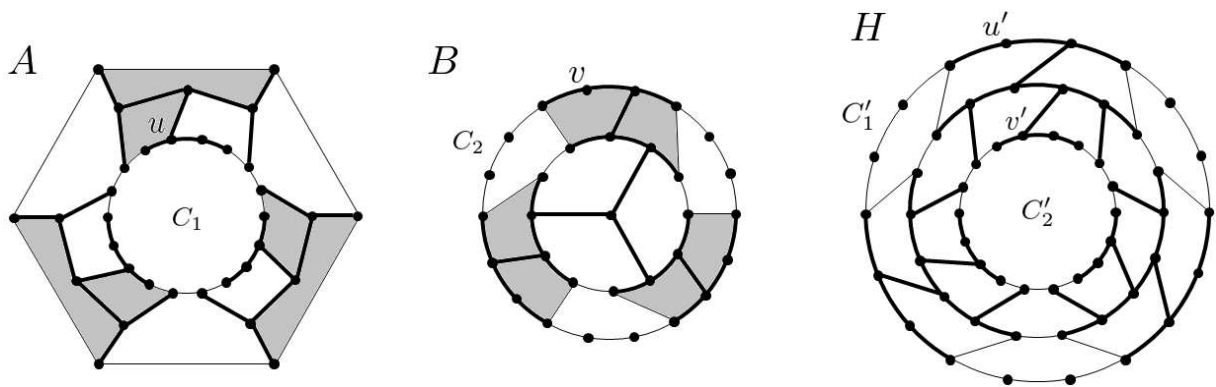}
 \caption{Plane graphs $A$, $B$ and $H$.}
 \label{figure1}
\end{figure}

\textbf{Remark:}
In the proof of Theorem \ref{constructionthm}, every facial cycle of the fullerene graph $F_k$ which is edge-disjoint with the defined Hist has length $6$. In contrast to $F_k$, the fullerene graph in Figure \ref{fullerene_fig} has facial cycles of length $5$ which are edge-disjoint with the illustrated Hist. By computer search, T. Jatschka \cite{J} showed that there are fullerene graphs with Hists with $38$ vertices and that every fullerene graph with less than $38$ vertices does not have a Hist. \\

A class of graphs similar to fullerene graphs are cubic hexangulations.
Recall that a \textit{hexangulation} of a surface is a 2-connected graph with an embedding on the surface such that every facial cycle has length $6$. For example, consider the dual of the triangulation in Figure \ref{hexafig}. Using this type of construction, we see that there are infinitely many bipartite cubic hexangulations $G$ of the torus with $|V(G)| \equiv \ 0  \,\,(\text{mod} \  4)$. Corollary \ref{HIST} directly shows that such hexangulations $G$ do not contain a Hist. 
We asked in the first version of this paper whether there are finitely or infinitely many hexangulations of the torus with a Hist. This question is answered in the next theorem.


\begin{thm}\label{constructhex}
There are infinitely many cubic hexangulations of the torus with a Hist.
\end{thm}

\begin{proof}
Let $G_0$ and $G_1$ be the hexangulations of the torus
shown in the left and the center of Figure \ref{figure2}. 
(The top and the bottom, the left and the right are identified, respectively.) 
Let $T$ be the hexangulation of the annulus 
shown in the right of Figure \ref{figure2}. 
(The top and the bottom are identified.)  
We construct the cubic hexangulation $G_k$ $(k\ge 2)$ of the torus
recursively, 
by 
(i) cutting $G_{k-1}$ along the cycle $v_1v_2v_3v_4v_5v_6$, and 
(ii) inserting $T$ with appropriate identification. 
Then $G_k$ is a cubic hexangulation with $(12k+10)$ vertices,
and has a Hist,
which is presented in Figure \ref{figure2} by the bold edges. \qed
\end{proof}

\begin{figure}
\centering
{\unitlength 0.1in%
\begin{picture}(12.6000,8.6000)(3.7000,-12.3000)%
%
\special{sh 1.000}%
\special{ia 400 400 30 30 0.0000000 6.2831853}%
\special{pn 20}%
\special{ar 400 400 30 30 0.0000000 6.2831853}%
%
\special{sh 1.000}%
\special{ia 800 400 30 30 0.0000000 6.2831853}%
\special{pn 20}%
\special{ar 800 400 30 30 0.0000000 6.2831853}%
%
\special{sh 1.000}%
\special{ia 800 800 30 30 0.0000000 6.2831853}%
\special{pn 20}%
\special{ar 800 800 30 30 0.0000000 6.2831853}%
%
\special{sh 1.000}%
\special{ia 400 800 30 30 0.0000000 6.2831853}%
\special{pn 20}%
\special{ar 400 800 30 30 0.0000000 6.2831853}%
%
\special{sh 1.000}%
\special{ia 400 1200 30 30 0.0000000 6.2831853}%
\special{pn 20}%
\special{ar 400 1200 30 30 0.0000000 6.2831853}%
%
\special{sh 1.000}%
\special{ia 800 1200 30 30 0.0000000 6.2831853}%
\special{pn 20}%
\special{ar 800 1200 30 30 0.0000000 6.2831853}%
%
\special{sh 1.000}%
\special{ia 1200 1200 30 30 0.0000000 6.2831853}%
\special{pn 20}%
\special{ar 1200 1200 30 30 0.0000000 6.2831853}%
%
\special{sh 1.000}%
\special{ia 1200 800 30 30 0.0000000 6.2831853}%
\special{pn 20}%
\special{ar 1200 800 30 30 0.0000000 6.2831853}%
%
\special{sh 1.000}%
\special{ia 1200 400 30 30 0.0000000 6.2831853}%
\special{pn 20}%
\special{ar 1200 400 30 30 0.0000000 6.2831853}%
%
\special{sh 1.000}%
\special{ia 1600 400 30 30 0.0000000 6.2831853}%
\special{pn 20}%
\special{ar 1600 400 30 30 0.0000000 6.2831853}%
%
\special{sh 1.000}%
\special{ia 1600 800 30 30 0.0000000 6.2831853}%
\special{pn 20}%
\special{ar 1600 800 30 30 0.0000000 6.2831853}%
%
\special{sh 1.000}%
\special{ia 1600 1200 30 30 0.0000000 6.2831853}%
\special{pn 20}%
\special{ar 1600 1200 30 30 0.0000000 6.2831853}%
%
\special{pn 20}%
\special{pa 1600 1200}%
\special{pa 1600 400}%
\special{fp}%
\special{pa 1600 400}%
\special{pa 400 400}%
\special{fp}%
\special{pa 400 400}%
\special{pa 400 1200}%
\special{fp}%
\special{pa 400 1200}%
\special{pa 1600 1200}%
\special{fp}%
\special{pa 1200 1200}%
\special{pa 1200 400}%
\special{fp}%
\special{pa 800 400}%
\special{pa 800 1200}%
\special{fp}%
\special{pa 800 1200}%
\special{pa 1600 400}%
\special{fp}%
\special{pa 1600 800}%
\special{pa 400 800}%
\special{fp}%
\special{pa 400 800}%
\special{pa 800 400}%
\special{fp}%
\special{pa 1200 400}%
\special{pa 400 1200}%
\special{fp}%
\special{pa 1200 1200}%
\special{pa 1600 800}%
\special{fp}%
\end{picture}}%
\caption{The dual of bipartite cubic hexangulations $G$ of the torus
satisfying $|V(G)| \equiv \ 0  \,\,(\text{mod} \  4)$. The top and the bottom, the left and the right are identified, respectively.
}
\label{hexafig}
\end{figure}
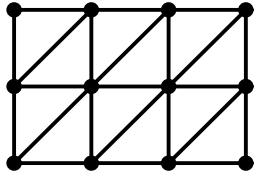

The length of a shortest non-contractible cycle of a graph embedded 
on a non-spherical surface
is called the \textit{edge-width} of the graph.
Note that $G_k$ in the proof of Theorem \ref{constructhex}
is bipartite for every $k\ge 0$,
has girth $6$ and edge-width exactly $6$ for every $k\ge 1$.

After submitting the first version of this paper,
the authors were informed that
Zhai, Wei, He and Ye \cite{WY} also proved Theorem \ref{constructhex},
together with the case of the Klein bottle.
It was also announced that their constructed hexangulations can have arbitrary large edge-width.



\begin{figure}[h]
 \centering
 \includegraphics[width=16cm]{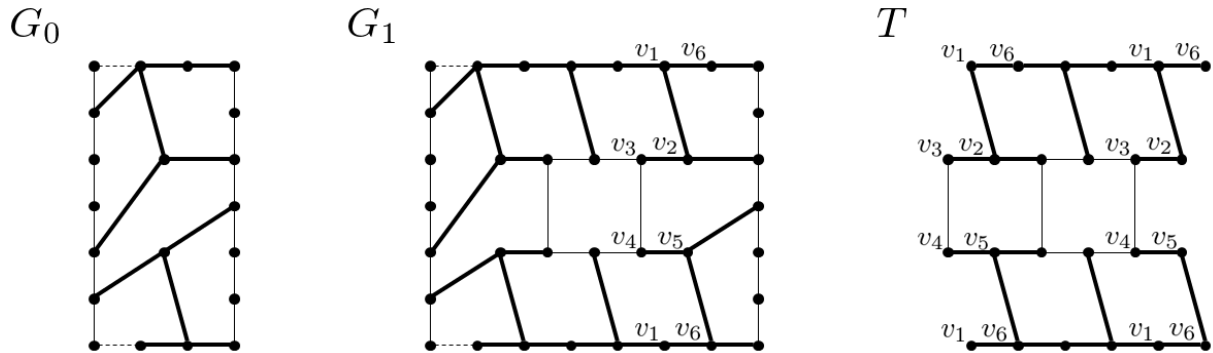}
 \caption{Hexangulations $G_0$, $G_1$ of the torus and $T$ of the annulus.}
 \label{figure2}
\end{figure}

\section*{Acknowledgments}

We are grateful to R.~Nedela, P.~Vr\'{a}na, and T.~Kaiser
for helpful discussions.We also thank Y.~Egawa for several useful comments which shortened the proof of Theorem \ref{inflation}. We are also grateful to the anonymous referees. In particular,
one of the referees suggested a proof of Proposition \ref{constlemma}
using vertex-transitive graphs. Part of the work was done
by the first and third authors during ``8th Workshop on the Matthews-Sumner Conjecture and Related Problems''
in Pilsen. They appreciate the organizers of the workshop for their hospitality.


\end{document}